\documentclass{ieeeconf}
\usepackage[T1]{fontenc}
\usepackage{microtype}

\usepackage{cite}

\usepackage{amssymb}
\usepackage{amsmath}
\usepackage{graphicx}

\newtheorem{theorem}{Theorem}
\newtheorem{lemma}{Lemma}
\newtheorem{corollary}{Corollary}
\newtheorem{definition}{Definition}
\DeclareMathOperator{\tr}{tr}
\DeclareMathOperator{\im}{im}
\DeclareMathOperator{\rk}{rank}

\newcommand{\dgp}[2]{%
\Delta(\,#1\,|\,#2\,)%
}

\newcommand*\dif{\mathop{}\!\mathrm{d}}

\begin{document}
\title{Submodularity of Energy Related \\Controllability Metrics}

\author{Fabrizio~L.~Cortesi,
        Tyler~H.~Summers, and John Lygeros%
\thanks{The authors are with the Automatic Control Laboratory, ETH Zurich, Switzerland.}
}

\maketitle

\begin{abstract}
The quantification of controllability and observability has recently received new interest in the context of large, complex networks of dynamical systems. A fundamental but computationally difficult problem is the placement or selection of actuators and sensors that optimize real-valued controllability and observability metrics of the network.
We show that several classes of energy related metrics associated with the controllability Gramian in linear dynamical systems have a strong structural property, called \emph{submodularity}. This property allows for an approximation guarantee by using a simple greedy heuristic for their maximization. The results are illustrated for randomly generated systems and for placement of power electronic actuators in a model of the European power grid. 
\end{abstract}

\IEEEpeerreviewmaketitle


\section{Introduction}
The concepts of controllability and observability have long been considered fundamental structural system properties since Kalman's seminal work over half a century ago \cite{kalman1959general}. Recently though they have seen renewed interest in the context of large, complex networks. 
There are many engineering examples such as the electric power grid or transportation networks, but also biological examples like neural or metabolic networks.
Many efforts have also been made to understand the complex dynamics in social and economic networks, where there is much interest in controlling the spreading of contagion, such as shocks (e.g., defaults) in financial networks \cite{gai2010contagion_finance}.

In the context of such networks, an interesting problem to consider before the design of a control law is the selection or placement of sensors and actuators in the network.  Many approaches for sensor and actuator selection problems have been proposed \cite{van2001review}, but the corresponding combinatorial optimization problems are generally considered to be computationally difficult. Here we will consider the combinatorial structure of these problems that allows approximation guarantees using metrics based on the controllability and observability Gramians.

In a prominent article \cite{liu2011controllability}, Kalman's rank criterion together with the concept of structural controllability \cite{lin1974structural} is used to find the minimal set of driver nodes, the smallest set of nodes that have to be controlled individually in order to achieve full control over the entire network.
But this results in a rather simplified and sometimes misleading view, as the rank criterion can only considers whether a system is completely controllable/observable, regardless of other considerations such as the energy required to actually drive the system around the state space. 

The controllability Gramian quantifies the energy required to move the system around the state space. One can consider several scalarizations, such as the maximum eigenvalue, trace and determinant of the inverse Gramian \cite{muller1972analysis}.
Recently \cite{tylerquote} used the minimum eigenvalue to select which nodes to control, and also a trade-off between control energy and the minimum number of driver nodes is presented. The Gramian can be computed by solving a Lyapunov equation. There are specialized methods\cite{hammarling1982numerical}, including recent approximations based on Quantized Tensor Trains \cite{nipdirect}, that enable computation for very large networks.

Although selecting subsets of possible actuators to maximize these metrics is typically computationally hard, some combinatorial structure can improve the situation. For example, the trace of the controllability Gramian was recently found to be a modular set function \cite{summers2013optimal}, which allows globally optimal subsets of actuators to be obtained. In the present paper, we identify what several other metrics of the controllability Gramian that have a strong structural property that give approximation guarantees. Our main result is that
\begin{enumerate}
\item $-\tr(W_c^{-1})$
\item $\log \det W_c$ 
\item $\rk(W_c)$
\end{enumerate} are submodular set functions.
Submodularity has been recognized in many areas in machine learning and computer science, such as \cite{krause2007near}. Together with \cite{summers2013optimal}, the present paper is among the first to study submodularity in controllability problems.

This paper is structured as follows:
Section~II recalls some basics of linear systems theory and discusses several energy related measures for controllability.
Section~III gives a brief introduction into submodular functions and their maximization under cardinality constraint. Then a framework for actuator selection in the context of set functions is presented, and we present our main results on submodularity of controllability metrics. Finally in section~V we illustrate the results on numerical examples.

\section{Energy-related Controllability Measures}
This section reviews several classes of energy-related controllability metrics for dynamical networks based on the controllability Gramian of a linear model for the network dynamics. Of course, real networks have nonlinear dynamics; however, the combinatorial properties of optimal sensor and actuator selection problems, which is the main topic of this paper, are not understood even for linear systems. Furthermore, linear models are often used to approximate nonlinear models near an equilibrium point and have been used as a starting point in many recent papers on controllability in networks.

Because controllability and observability are dual properties \cite{kalman1959general}, we will only discuss controllability; analogous results and interpretations for observability follow directly.

\subsection{Linear System Dynamics}
We consider linear, continuous time, time-invariant dynamical systems of the form
\begin{equation}\label{eq:sys}
\begin{aligned}
\dot{x}(t) &= Ax(t) + Bu(t)
\end{aligned}
\end{equation}
where $A\in \mathbb{R}^{n\times n}$ is the dynamic matrix and $B\in \mathbb{R}^{n\times m}$ is the input matrix. The basic problem we consider can be formulated as follows. Let $V = \{ b_1,...,b_M\} $ be a set of possible columns corresponding to actuators that could be selected to form the input matrix. We consider the set function optimization problem
\begin{equation}\label{eq:maxsubmod}
\max_{S\subseteq V} f(S)\qquad \text{subject to } |S|\leq k,
\end{equation}
where $f : 2^V \rightarrow \mathbb{R}$ is a metric that quantifies to the controllability of the system when choosing subset $S \subseteq V$.

Actuators in real systems are usually energy limited, so an important class of metrics of controllability deals with the amount of input energy required to reach a given state from the origin. In particular, we can pose the following optimal control problem seeking the minimum energy input that drives the system from the origin to a final state $x_f$ at time $t$:
\begin{equation} \label{optprob}
\begin{aligned}
 \underset{{u(\cdot) \in \mathcal{L}_2  }}{\text{minimize}} & &&\int_0^t \Vert u(\tau) \Vert^2 d\tau \\
 \text{subject to} & && \dot{x}(t) = Ax(t) + Bu(t), \\
 			  & && x(0) = 0, \quad x(t) = x_f
 \end{aligned}
\end{equation}
Standard methods from optimal control theory can be used to derive the solution. 
If the system is controllable, the optimal input has the form 
\begin{equation}
\begin{aligned}
u^*(\tau) = B^T e^{A^T(t-\tau)} &\left( \int_0^t e^{A \sigma} B B^T e^{A^T \sigma} d\sigma \right)^{-1} x_f, \\
 &0 \leq \tau \leq t
\end{aligned}
\end{equation}
and the resulting minimum energy is
\begin{equation}
\int_0^t \Vert u^*(\tau) \Vert^2 d\tau = x_f^T \left(  \int_0^t e^{A \sigma} B B^T e^{A^T \sigma} d\sigma \right)^{-1} x_f.
\end{equation}
The symmetric positive semidefinite matrix
\begin{equation}
W_c(t) = \int_0^t e^{A \tau} B B^T e^{A^T \tau} d\tau \ \in \mathbf{R}^{n\times n}
\end{equation}
is called the \emph{controllability Gramian} at time $t$. The controllability Gramian has the same rank as $[B, AB, ..., A^{n-1} B]$ and defines an ellipsoid in the state space
\begin{equation}
\mathcal{E}_{min}(t) = \left\{ x \in \mathbf{R}^n \mid  x^T W_c(t)^{-1} x \leq 1 \right\}
\end{equation}
that contains the set of states reachable in $t$ seconds with one unit or less of input energy. The eigenvectors and corresponding eigenvalues of $W_c$ quantify energy required to move the system in various state space directions.

For stable systems, the state transition matrix $e^{At}$ consists of decaying exponentials, so a finite positive definite limit of the controllability Gramian always exists and is given by
\begin{equation}
W_c = \int_0^\infty e^{A \tau} B B^T e^{A^T \tau} d\tau \ \in \mathbf{R}^{n\times n}
\end{equation}
This matrix defines an ellipsoid in the state space that gives the states reachable with one unit or less of energy, regardless of time. This infinite-horizon controllability Gramian can be computed by solving a Lyapunov equation
\begin{equation}
A W_c + W_c A^T + B B^T = 0,
\end{equation}
which is a system of linear equations and is therefore easily solvable, even for large systems. Specialized algorithms have been developed to compute the solution \cite{hammarling1982numerical}. For unstable system, an alternative definition of the controllability Gramian can be used \cite{balancedgramian}.

We still need scalarize $W_c$, which is a positive semidefinite matrix. We want $W_c$ ``large'' so that $W_c^{-1}$ is ``small'', requiring small amount of input energy to move around the state space. There are a number of possible metrics for the size of $W_c$, several of which we now discuss.

\subsubsection{$\mathbf{tr}(W_c)$}
The trace of the Gramian is inversely related to the average energy and can be interpreted as the average controllability in all directions in the state space. It is also closely related to the system $H_2$ norm:

\begin{equation}
\begin{aligned}
\Vert H \Vert_2^2  
			     &= \textbf{tr} \left( C \int_0^\infty e^{A t} B B^T e^{A^T t} dt C^T \right) \\
			     &= \textbf{tr} (C W_c C^T)
\end{aligned}
\end{equation}
i.e., the system $H_2$ norm is a weighted trace of the controllability Gramian.

%

\subsubsection{$\mathbf{tr}(W_c^{-1})$}
The trace of the inverse controllability Gramian is proprtional to the energy needed on average to move the system around on the state space. Note that when the system is uncontrollable, the inverse Gramian does not exist and the average energy is infinite because there is at least one direction in which it is impossible to move the system using the inputs. In this case, one could consider the trace of the pseudoinverse, which is the average energy required to move the system around the controllable subspace. 


\subsubsection{$\log \det W_c$}
The determinant of the Gramian is related to the volume enclosed by the ellipse it defines
\[V(\mathcal{E}_{min}) = \frac{\pi^{n/2}}{\Gamma(n/2+1)}\sqrt[n]{\det W_c},\]
where $\Gamma$ is the Gamma function. This means that the determinant is a volumetric measure of the set of states that can be reached with one unit or less of input energy. Since determinant is numerically problematic in high dimensions, and because the matrix logarithm is a monotone matrix function that preserves the semidefinite order, we will consider optimizing $\log \det W_c$. Note that for uncontrollable systems, the ellipsoid volume is zero, so $\log \det W_c = -\infty$. In this case, one could consider the associated volume in the controllable subspace. 



\subsubsection{$\lambda_{min} (W_c)$}

The smallest eigenvalue of the Gramian is a worst-case metric inversely related to the amount of energy required to move the system in the direction in the state space that is most difficult to control. 

\subsubsection{$\rk(W_c$)} The rank of the Gramian is the dimension of the controllable subspace.

\section{Optimal Actuator Placement}
In this section, we briefly review submodularity and consider which of the above controllability metrics are submodular, which provides approximation guarantees for associated actuator selection problems. Detailed treatments of the topic are provided by many texts \cite{lovasz83,fujishige2005submodular,schrijver2003combinatorial,iwata2008submodular}, as well as by recent survey \cite{google} which also focuses on maximization.

\subsection{Submodularity Basics}
A set function $f\colon 2^V \to \mathbb{R}$ assigns a real value to every subset $S\subseteq V$.
Submodularity is defined as follows. 
\begin{definition}\label{def:submod}
A set function $f\colon 2^V \to \mathbb{R}$ is submodular if and only if
\[f(A\cup \{a\})-f(A)\geq f(B\cup \{a\})-f(B),\]
for any subsets $A\subseteq B\subseteq V$ and $\{a\}\in V\backslash	B$.
\end{definition}
This is a diminishing gains property: adding an element $a$ to a larger set $B$ will result in a smaller gain than adding the same element to a subset $A$. We will denote this the marginal gain as $\dgp{a}{S}=f(S\cup\{a\})-f(S).$

A useful consequence of this definition is expressed in the following theorem
\begin{theorem}
A set function $f: 2^V \rightarrow \mathbf{R}$ is submodular if and only if the derived set functions $f_a : 2^{V - \{ a \} } \rightarrow \mathbf{R}$
$$f_a (X) = f(X \cup \{a\}) - f(X)  $$ 
are monotone decreasing for all $a \in V$.
\end{theorem}
A set function $f_a(S)$ is called monotone decreasing if for all $S_1$, $S_2\subseteq V$ holds, that
\[S_1\subseteq S_2 \quad \Leftrightarrow\quad f_a(S_1)\geq f_a(S_2).\]

\subsection{Submodular Function Maximization}
We consider set function optimization problems of the form
\begin{equation}\label{eq:maxsubmod}
\max_{S\subseteq V} f(S)\qquad \text{subject to } |S|\leq k.
\end{equation}
Maximization of monotone submodular functions is NP-hard, but the so-called greedy heuristic can be used to obtain a solution that is provably close to the optimal solution. The greedy algorithm for \eqref{optprob} starts with an empty set, $S_0 \leftarrow \emptyset$, computes the gain $\dgp{a}{S_i} = f(S_{i}\cup \{a\})-f(S_{i})$ for all elements $a\in V\backslash S_{i}$ and adds the element with the highest gain:
\[S_{i+1} \leftarrow S_{i}\cup \{\arg \max_{a} \dgp{a}{S_i}\; |\; a\in V\backslash S_{i}\}. \]
The algorithm terminates after $k$ iterations.

Performance for a submodular, non-negative, monotone set function is guaranteed by the well known bound~\cite{greedybound}:
\begin{equation}
	\dfrac{f(S_{greedy})}{f(S_{optimal})}
	\geq 1-\left(\frac{k-1}{k}\right)^k
	\geq \frac{e-1}{e}
	\approx 0.63,
\label{eq:greedy_bound}\end{equation}
where $S_{optimal}$ is the optimal subset and $S_{greedy}$ is the subset obtained by the greedy algorithm. This is the best any polynomial time algorithm can achieve \cite{feige1998threshold}, assuming $P\neq NP$. Note that this is a worst-case bound; the greedy algorithm often performs much better than the bound in practice. 

The space of symmetric $n\times n$ matrices $\mathcal{S}^n$ has a partial semidefinite order: $W_1\succeq W_2$ if  $W_1-W_2\succeq 0$.
The space of symmetric positive definite matrices is denoted $\mathcal{S}_{++}^n$ and the space of symmetric positive semidefinite matrices is denoted $\mathcal{S}^n_{+}$. 

We now demonstrate the submodularity of several classes of controllability metrics involving functions of the controllability Gramian.

\subsection{Trace of the inverse Gramian}
Suppose $A \in \mathbf{R}^{n\times n}$ is a stable system dynamics matrix and $V = \{b_1,..., b_M \} $ is a set of possible columns that can be used to form or modify the system input matrix $B$. The problem is to choose a subset of $V$ to maximize a metric of controllability. For a given $S \subseteq V$, we form $B_S = [B_0 \quad b_s]$ given a (possibly empty) existing matrix $B_0$ and using the associated columns defined by  $s \in S$ and the associated controllability Gramian $W_S  = \int_0^\infty e^{A \tau} B_S B_S^T e^{A^T \tau} d\tau =\sum_{s\in S}W_{\{s\}}$, which is the unique positive semidefinite solution the Lyapunov equation
\begin{equation}
AW_S + W_S A^T + B_S B_S^T = 0.
\end{equation}
To simplify notation, we write $W_s$ for $W_{\{s\} }$. 

We first consider the trace of the inverse of the controllability Gramian. We assume in this subsection that for any $S \subseteq V$ the associated Gramian $W_S$ is invertible. This is the case, for example, if the network already has a set of actuators that provide controllability and we would like to add additional actuators to improve controllability. We discuss how to deal with non-invertibility of the Gramian subsequently.

\begin{theorem}\label{theorem:trace}
Let $V = \{ b_1,...,b_M\}$ be a set of possible input matrix columns and $W_S$ the controllability Gramian associated with $S \subseteq V$. The set function $f : 2^{V}\to \mathbf{R}$ defined as
\[f(S) = -\tr(W_{S}^{-1}) \]
 is submodular.
\end{theorem}

Before proving Theorem \ref{theorem:trace}, we need a result on differentiable matrix functions that are monotone with respect to the semidefinite order.

\begin{definition}\label{def:mon}
A matrix function $f : \mathbf{S}^n_{++}\to\mathbf{R}$ is called monotone decreasing if
\[W_1\preceq W_2 \quad \Rightarrow \quad f(W_1)\geq f(W_2) \qquad \forall\, W_1, W_2 \in \mathbf{S}^n_{++}.\]
\end{definition}

\begin{lemma}\label{theorem:mono} Let $f\colon\mathbf{S}^n_{++}\to\mathbf{R}$ be continuously differentiable matrix function and $X(t)=A+tB$ a ray in the space of symmetric positive definite matrices with arbitrary $A \in \mathbf{S}^n_{++}$, $B\in \mathbf{S}^n_+$, and $t > 0$.
Then $f$ is monotone decreasing if
\begin{equation}B\succeq 0 \quad \Rightarrow \quad \frac{\dif}{\dif t}f(X(t))\leq 0\qquad \forall t.\label{eq:mono}
\end{equation}
\end{lemma}
\textit{Proof:} Take any $W_1$,~$W_2\in\mathbf{S}^n_{++}$ and assume $W_1\preceq W_2$. Then there exist $t_1, t_2\in \mathbf{R}$, $A \succ 0$, $B \succeq 0$, and $t_1\leq t_2$  such that $X(t_1)=W_1$ and $X(t_2)=W_2$.
If $B\succeq 0$ implies $\frac{\dif}{\dif t}f(X(t))\leq 0$, then from
\begin{equation*}
 f(X(t_2)) = f(X(t_1)) +\int_{t_1}^{t_2}\!\frac{\dif}{\dif t}f(X(t))\dif t
\end{equation*}
it follows that $f(W_1)\geq f(W_2)$.

\textit{Proof of Theorem \ref{theorem:trace}:} 
Consider the derived set functions $f_a: 2^{V-\{a\}} \to \mathbf{R}$ 
\begin{equation*}
\begin{aligned}
f_a(S) &=  -\tr((W_{S\cup \{a\} })^{-1}) + \tr((W_{S})^{-1}) \\
            &= -\tr((W_{S}+W_{a})^{-1}) + \tr((W_{S})^{-1}).
\end{aligned}
\end{equation*}
and the derived matrix functions $\hat{f}_a :  \mathbf{S}_{++}^n \to \mathbf{R}$ defined by $\hat{f}_a(W_{S}) = f_a(S).$
Since $S_1\subseteq S_2\Leftrightarrow W_{S_1} \preceq W_{S_2}$, it is  clear that $f_a(S_1)\geq f_a(S_2)\Leftrightarrow \hat{f}_a(W_{S_1}) \geq \hat{f}_a(W_{S_2})$ and therefore if $\hat{f}_a$ is monotone decreasing, then so is $f_a$.

Let $X(t)=A+tB$ with $A\succ 0$, $B\succeq 0$ and $t>0$. We have
\begin{equation*}
\begin{aligned}
&\frac{\dif}{\dif t}\hat{f}_a\left(X(t)\right) \\
&=-\frac{\dif }{\dif t}\left[ \tr((A+tB+W_{a})^{-1}) + \tr((A+tB)^{-1}) \right]\\
&=\tr\left((A+tB+W_{a})^{-1}B(A+tB+W_{a})^{-1} \right) \\
& \quad - \tr \left((A+tB)^{-1}B(A+tB)^{-1}\right)\\
&=\tr\bigg( \left((A+tB+W_{a})^{-2}-(A+tB)^{-2}\right)B \bigg)\,.
\end{aligned}
\end{equation*}
where we used a standard matrix derivative formula and the cyclic property of trace to obtain the second and third equalities. Since $A+tB+W_{a}\succeq A+tB$, it follows that $(A+tB+W_{a})^{-2}\preceq(A+tB)^{-2}$. Thus,
\[\Gamma=(A+tB+W_{a})^{-2}-(A+tB)^{-2}\preceq 0.\]
Since the trace of the product of a positive and negative semidefinite matrix is non-positive, we have
\[\frac{\dif}{\dif t}\hat{f}_a\left(X(t)\right)=\tr(\Gamma B)\leq 0.\]
Thus, $\hat{f}_a$ and hence $f_a$ are monotone decreasing, and $f$ is submodular by Theorem~\ref{theorem:mono}. 

\subsection{Log determinant of the Gramian}
We now consider the log determinant of the controllability Gramian. We assume again that for any $S \subseteq V$ the associated Gramian is invertible. We have the following result. 
\begin{theorem}\label{theorem:determinant}
Let $V = \{ b_1,...,b_M\}$ be a set of possible input matrix columns and $W_S$ the controllability Gramian associated with $S \subseteq V$. The set function $f\colon 2^{V}\to \mathbf{R}$, defined as
\[f(S) = \log \det W_{S}  \]
is submodular.
\end{theorem}
\begin{proof} The proof uses the same idea as before, namely, showing monotonicity of a family of derived matrix functions. Consider the derived set functions $f_a: 2^{V-\{a\}} \to \mathbf{R}$
\begin{align*}
 f_a(S) &=  \log \det  W_{S \cup \{a\} } - \log \det W_{S} \\
 &=  \log\det (W_{S}+W_{a}) - \log\det W_{S} 
\end{align*}
and the associated matrix functions $\hat{f}_a :  \mathbf{S}_{++}^n \to \mathbf{R}$ defined by $\hat{f}_a(W_{S}) = f_a(S).$

Let $X(t)=A+tB$ with $A\succ 0$, $B\succeq 0$ and $t>0$. We have
\begin{align*}
\frac{\dif}{\dif t} &\hat{f}_a(X(t)) \\
 &=\frac{\dif}{\dif t}\left[ \log \det (A+tB + W_a) - \log \det (A+tB) \right] \\
 &=\tr((A+tB+W_{a})^{-1}B)-\tr((A+tB)^{-1}B)\\
 &=\tr(((A+tB+W_{a})^{-1}-(A+tB)^{-1})B)\\
 &\leq 0.
\end{align*}
where we used again a standard matrix derivative formula and the cyclic property of trace to obtain the second and third equalities.  The remainder of the proof follows the previous proof. 
\end{proof}

\vspace{\baselineskip}
\begin{corollary}
The related set function $g: 2^V \rightarrow \mathbf{R}$ defined by $g(S)=\log\sqrt[n]{\det{W_S}}$ is submodular.
\end{corollary}
\textit{Proof:} We have
\[g(S)= \frac{1}{n} \log \det W_S\]
Thus, from Theorem \ref{theorem:determinant} $g$ is a non-negatively scaled version of a submodular function and therefore also submodular.

Not all directions in the state space may be of equal importance, one might want to use a weight matrix as an additional design parameter for an actuator selection problem. In a simple case, the weight matrix could be a diagonal matrix, assigning a relative weight to every state. We have the following corollary; the proof follows exactly the same argument as in the previous theorem.

\begin{corollary}
Let $V = \{ b_1,...,b_M\}$ be a set of possible input matrix columns and $W_S$ the controllability Gramian associated with $S \subseteq V$. The set function  $g\colon 2^V \to \mathbf{R}$ defined as
\[g(S)=\log\det(QW_SQ^T),\]
where  $Q\in\mathbf{R}^{m\times n}$ with $m\leq n$ and $\rk(Q)=m$, is submodular.
\end{corollary}


\subsection{Rank of the Gramian}
The controllability metrics $-\tr(W_S^{-1})$ and $\log \det W_S$ fail to distinguish amongst subsets of actuators that do not yield a fully controllable system. In particular, these functions are undefined, or are interpreted to return $-\infty$, when the Gramian is not full rank. One way to handle these cases is to first consider the rank of the controllability Gramian. The following result shows that this is also a submodular set function. 
\begin{theorem}
Let $V = \{ b_1,...,b_M\}$ be a set of possible input matrix columns and $W_S$ the controllability Gramian associated with $S \subseteq V$. The set function $f\colon 2^V \to \mathbf{R}$, defined as
\[f(S) = \rk(W_S) \]
is submodular.
\end{theorem}

\begin{proof} For two linear transformations $V_1$,~$V_2 \in \mathbf{R}^{n\times n}$, we have
\begin{equation*}
\begin{aligned}
\rk & (V_1+V_2) \\ &= \rk(V_1) + \rk(V_2) -\dim(\im(V_1) \cap \im(V_2)).
\end{aligned}
\end{equation*}
We can form gain functions $f_a: 2^{V-\{a\} } \rightarrow \mathbf{R}$
\begin{equation}
\begin{aligned}
f_a(S)&=\rk(W_{S \cup \{a\} }) - \rk (W_S) \\
           &=\rk(W_{a}) - \dim(\im(W_S)\cap\im(W_{a}))
\end{aligned}
\end{equation}
It is now easy to see that $f_a$ is monotone decreasing: the first term in the second line is constant and the second term decreases because $\im(W_S)$ only increases with $S$.
\end{proof}

Another way to handle uncontrollable systems is to work with related continuous metrics defined for uncontrollable systems, such as the trace of the pseudoinverse $\tr (W_S^\dag)$, which corresponds to the average energy required to move the system around the controllable subspace, or the log product of non-zero eigenvalues $\log\Pi_{i=1}^{\rk W_S} \lambda_i(W_S)$, which relates to the volume of the subspace reachable with one unit of input energy.

\subsection{Smallest eigenvalue of the Gramian}
We have seen so far that the trace of the Gramian is a modular (and thus both sub- and supermodular) set function of actuator subsets and that the trace of the inverse Gramian and the log determinant of the Gramian are submodular set functions. These functions are also all concave matrix functions of the Gramian. Given the connections between submodular functions and concave functions, one might be tempted to conjecture that any concave function of the Gramian, e.g. the smallest eigenvalue $\lambda_{min}(W_S)$, corresponds to a submodular function of actuator subsets. However, we now show by counterexample that this is false. 

We show an example where this function violates the diminishing gains property of a set function $f(S)$
\begin{equation*}\label{eq:supermodular}
\dgp{s}{A}\geq\dgp{s}{B}, \quad A\subseteq B \subseteq V,\; s\notin B,
\end{equation*}
where $\dgp{s}{A} = f(A \cup \{s\}) - f(A).$ Consider the system defined by
\[A = \begin{bmatrix} -8 &0 & -2\\
0 & -2 & -8\\
7& 0 & -3
\end{bmatrix},\qquad B_V=[b_V]=I_3.\]
We see that the diminishing returns  property holds in some cases, e.g.,
\[\dgp{b_3}{\{b_1\}}=0.037 \geq 0.033 = \dgp{b_3}{\{b_1,b_2\}},\]
but is violated in others
\[\dgp{b_3}{\{b_2\}}=0.001 \leq 0.033 = \dgp{b_3}{\{b_1,b_2\}}.\]


\section{Illustrative Numerical Examples}\label{sec:examples}
In this section, we illustrate the results on randomly generated systems and for placement of power electronic actuators in a model of the European power grid. The problem data is a system dynamics matrix $A \in \mathbf{R}^{n\times n}$, a set of possible input matrix columns $V = \{ b_1,...,b_M \}$, and an integer number $k$ of actuators to choose from this set to form an input matrix that maximizes a controllability metric.

Since the trace of the inverse Gramian and the log determinant do not distinguish amongst actuator subsets that result in uncontrollable systems, we use for these metrics a two-stage greedy algorithm that first greedily optimizes the rank of the Gramian until controllability is achieved. During this stage, if two or more actuators yield a Gramian with the same rank, the trace of the pseudoinverse $\tr (W_S^\dag)$ or the log product of non-zero eigenvalues $\log\Pi_{i=1}^{\rk W_S} \lambda_i(W_S)$ is used to determine preference. After controllability it achieved, the trace of the inverse and log determinant are used in a second stage until the desired number of actuators has been selected..

We assume that there exists a subset of $V$ of size $k$ that renders the system fully controllable. However, although a real-valued controllability metric is optimized, the greedy algorithm does not necessarily guarantee that full controllability is achieved after $k$ iterations; if and how rank constraints could be incorporated is an interesting topic for future work.

\subsection{Greedy performance on a random system}
To evaluate performance of the greedy algorithm and to compare the various controllability metrics, we first consider randomly generated data. We use Matlab's \texttt{rss} routine to generate a stable dynamics matrix with random stable eigenvalues. We use $V=\{e_1,...,e_n\}$, where $e_i$ is the $i$th unit vector in $\mathbf{R}^n$, i.e., we assume one can choose states in which a control input can be injected.

Figure~\ref{fig:greedy_spread} shows the result of applying the greedy algorithm to maximize the log determinant metric with $n=25$ and $k=7$. This problem is small enough to evaluate every possible 7-element actuator subset, and this result is also shown in a histogram. 
Our algorithm finds a good set $S_{greedy}$ scoring
\[\frac{f(S_{greedy})}{f(S_{opt})}\approx98\%\]
of the absolute optimum value, which is better than 99.93\% of all other possible choices.
Those few actuator choices that score better though, are significantly better with respect to the reachable subspace volume, where we only score  
\[\frac{V_{greedy}}{V_{opt}}
=\sqrt{\frac{e^{f(S_{greedy})}}{e^{f(S_{opt})}}} \approx 68.1\%.\]
Similar results are obtained with other randomly generated data.
\begin{figure}
\centering
\includegraphics[width=\linewidth,height=4cm]{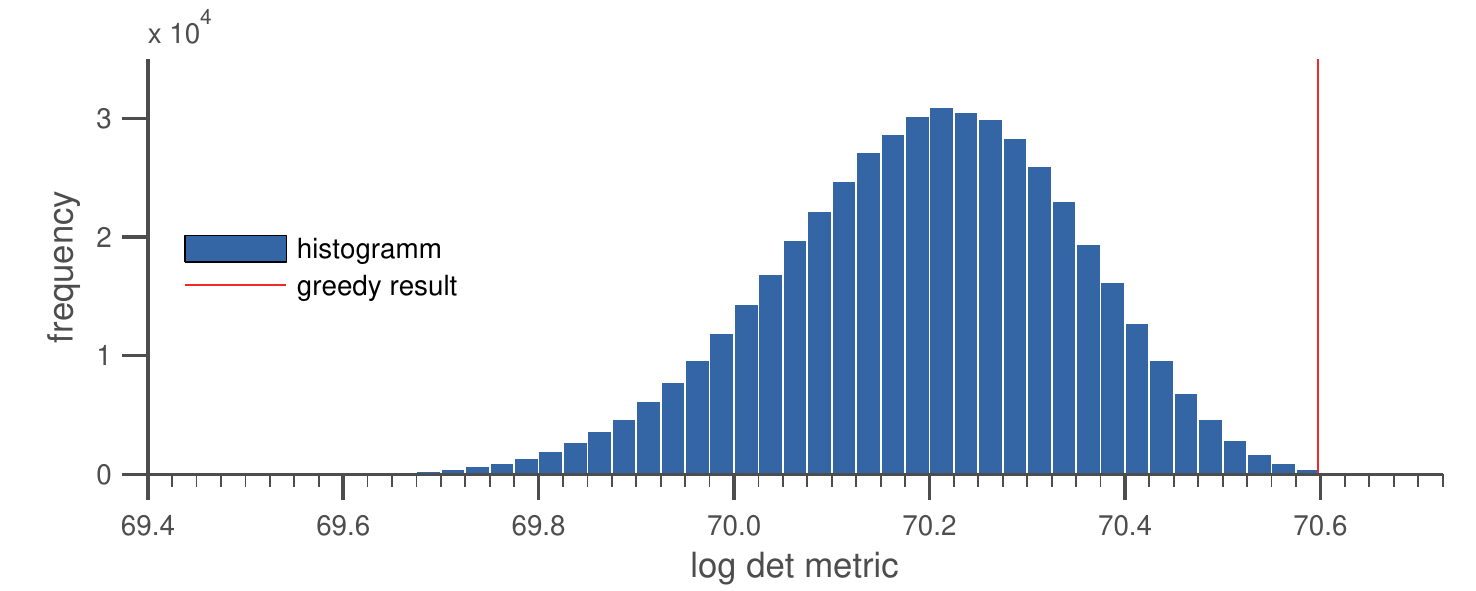}
\caption{\label{fig:greedy_spread}Histogram displaying the positive log determinant metric $f'(S)$ for all possible selection of 7 actuators from a set of possible 25. The result achieved by greedy optimization is displayed by the red line.}
\end{figure}

\subsection{Power Electronic Actuator Selection}
In this second example we consider the placement of HVDC (high voltage direct current) links in a simplified model of the Continental European power grid. HVDC lines connect two points in the alternating current grid via a direct current line which allows the injection of both active and reactive power into the grid on both ends of the line. They can be used to increase transport capacity and to improve transient stability, power oscillation damping, and voltage stability control \cite{hvdc}.

The system dynamics are based on the well known swing equations, which describe the generator rotor dynamics \cite{bookPowerSystems}.
The model consists of 74 buses which are each connected to a generator and a constant impedance load. The system is based on a midday operation snapshot of the actual grid and a simplification of the model presented in \cite{haase}. The HVDC links are modeled as two ideal current sources, one on each end, allowing for three degrees of freedom to influence the frequency dynamics at the terminals; for modeling details see \cite{g3,g1,g2}. The swing equations are linearized around a nominal operating point. Each actuator results in a slightly different linearized dynamics matrix, but all are within 1\% in Frobenius norm. Therefore, we use a constant $A$ taken from one of the linearizations.
\begin{figure}
        \centering
                \includegraphics[width=\linewidth]{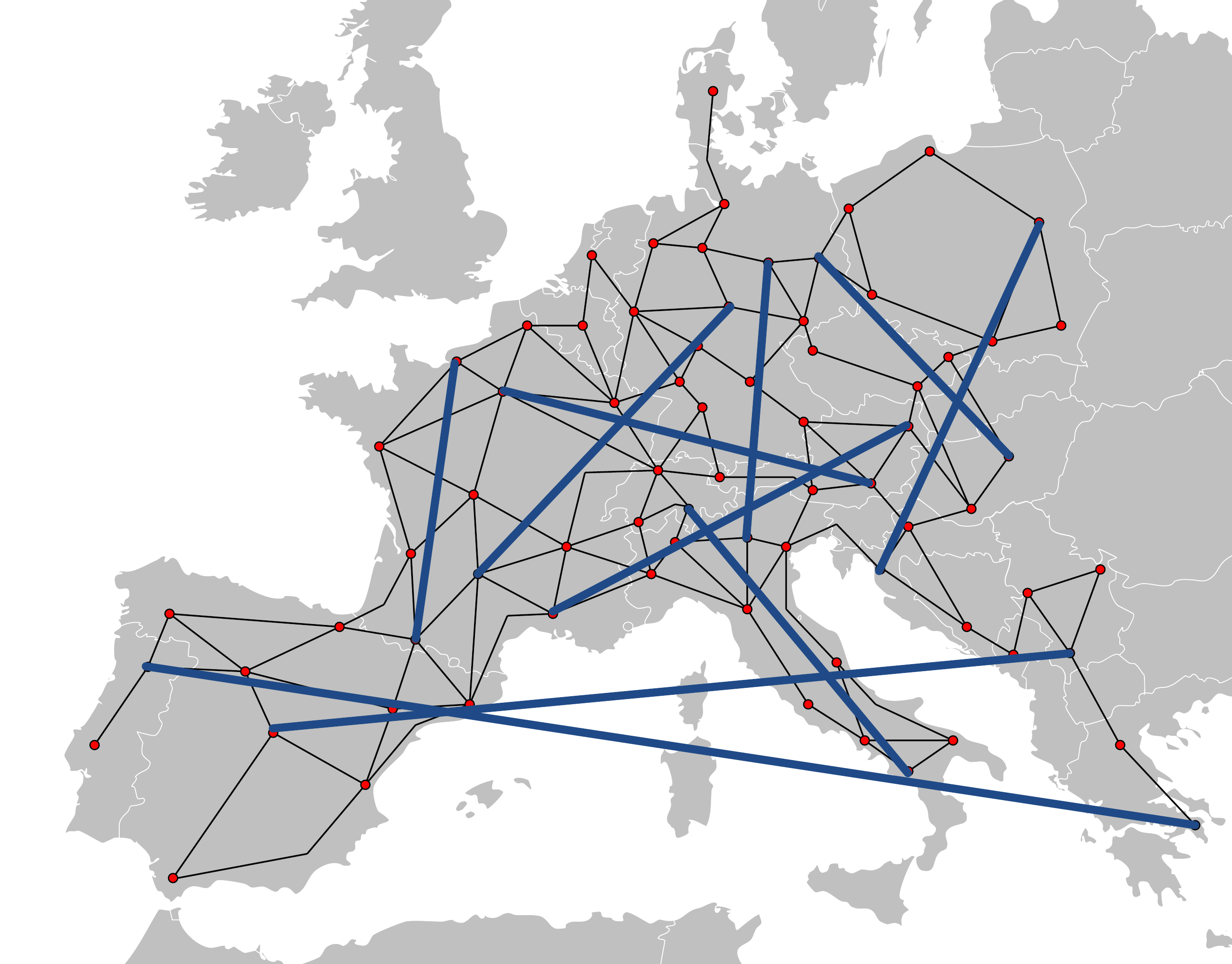}
                \caption{Map of the European power grid model, red nodes represent the buses that are interconnected via the blue HVDC lines placed to maximise $\log \det W_c$.}
                \label{fig:map_res}
\end{figure}
An HVDC link can connect any of the 74 nodes leaving a choice of 2701 possible HVDC locations.  Choosing a subset of size 10 results in approximately $5.6 \times 10^{27}$ possible combinations, which is far too many for brute force evaluation. 

Figure \ref{fig:map_res} show the placement obtained by using the two-stage greedy algorithm with the log determinant metric. Compared to the trace metric used in \cite{summers2013optimal}, we see that the lines are in general longer, connecting buses that are further apart, and more evenly distributed in the network. Also, no node is part of more than one HVDC line. This is due to focus of the trace metric mainly on the few directions in which controllability can be increased the most, whereas the determinant is focused more on the overall performance in all directions.

We emphasize that this model does not consider any other factors such as political, geographical or financial aspects. The greedy algorithm was implemented in Matlab, building upon the Submodular Function Optimization toolbox \cite{toolbox}.

\section{Conclusion}
We have considered optimal actuator placement problems in complex dynamical networks showed that several important class of metrics related to the controllability Gramians, viz. the log determinant and rank for the Gramian, and the trace of the inverse Gramian, result in submodular set functions. This allows approximation guarantees to be obtained with a simple greedy algorithm. By duality, all of the results hold for corresponding sensor placement problems using metrics of the observability Gramian. The results were illustrated in problems with randomly generated data and via placement of power electronic actuators in a model of the European power grid. 

\section*{Acknowledgment}
The authors would like to thank Dr. Alexander Fuchs for providing details and helpful discussion about the power grid model discussed in section~\ref{sec:examples}.


\bibliographystyle{plain}
\bibliography{bibtex/ref}

\begin{thebibliography}{10}

\bibitem{feige1998threshold}
U.~Feige.
\newblock A threshold of ln n for approximating set cover.
\newblock {\em Journal of the ACM (JACM)}, 45(4):634--652, 1998.

\bibitem{g3}
A.N. Fuchs, S.~Mari{\'e}thoz, M.~Larsson, and M.~Morari.
\newblock {Grid stabilization through VSC-HVDC using wide area measurements}.
\newblock In {\em IEEE Powertech, Power System Technology}, Trondheim, Norway,
  June 2011.

\bibitem{g1}
A.N. Fuchs and M.~Morari.
\newblock {Actuator performance evaluation using LMIs for optimal HVDC
  placement}.
\newblock In {\em European Control Conference}, Zurich, Switzerland, July 2013.

\bibitem{g2}
A.N. Fuchs and M.~Morari.
\newblock {Placement of HVDC links for power grid stabilization during
  transients}.
\newblock In {\em IEEE Powertech, Power System Technology}, Grenoble, France,
  June 2013.

\bibitem{fujishige2005submodular}
Satoru Fujishige.
\newblock {\em Submodular functions and optimization}, volume~58.
\newblock Elsevier, 2005.

\bibitem{gai2010contagion_finance}
Prasanna Gai and Sujit Kapadia.
\newblock Contagion in financial networks.
\newblock {\em Proceedings of the Royal Society A: Mathematical, Physical and
  Engineering Science}, 466(2120):2401--2423, 2010.

\bibitem{haase}
T.~Haase.
\newblock {\em Anforderungen an eine durch Erneuerbare Energien gepr{\"a}gte
  Energieversorgung: Untersuchung des Regelverhaltens von Kraftwerken und
  Verbundnetzen}.
\newblock PhD thesis, 2006.

\bibitem{hammarling1982numerical}
S.J. Hammarling.
\newblock Numerical solution of the stable, non-negative definite lyapunov
  equation.
\newblock {\em IMA Journal of Numerical Analysis}, 2(3):303--323, 1982.

\bibitem{iwata2008submodular}
S.~Iwata.
\newblock Submodular function minimization.
\newblock {\em Mathematical Programming}, 112(1):45--64, 2008.

\bibitem{kalman1959general}
R.~Kalman.
\newblock On the general theory of control systems.
\newblock {\em IRE Transactions on Automatic Control}, 4(3):110--110, 1959.

\bibitem{toolbox}
A.~Krause.
\newblock Sfo: A toolbox for submodular function optimization.
\newblock {\em The Journal of Machine Learning Research}, 11:1141--1144, 2010.

\bibitem{google}
A.~Krause and D.~Golovin.
\newblock Submodular function maximization.
\newblock {\em Tractability: Practical Approaches to Hard Problems}, 3, 2012.

\bibitem{krause2007near}
A.~Krause and C.~Guestrin.
\newblock Near-optimal observation selection using submodular functions.
\newblock In {\em AAAI}, volume~7, pages 1650--1654, 2007.

\bibitem{hvdc}
A.~L'Abbate, G.~Migliavacca, U.~Hager, C.~Rehtanz, S.~Ruberg, H.~Ferreira,
  G.~Fulli, and A.~Purvins.
\newblock The role of facts and hvdc in the future paneuropean transmission
  system development.
\newblock In {\em 9th IET International Conference on AC and DC Power
  Transmission}, pages 1--8. IET, 2010.

\bibitem{lin1974structural}
C.-T. Lin.
\newblock Structural controllability.
\newblock {\em Automatic Control, IEEE Transactions on}, 19(3):201--208, 1974.

\bibitem{liu2011controllability}
Y.-Y. Liu, J.-J. Slotine, and A.~Barab{\'a}si.
\newblock Controllability of complex networks.
\newblock {\em Nature}, 473(7346):167--173, 2011.

\bibitem{lovasz83}
L.~Lov{\'a}sz.
\newblock Submodular functions and convexity.
\newblock In {\em Mathematical Programming The State of the Art}, pages
  235--257. Springer, 1983.

\bibitem{bookPowerSystems}
J.~Machowski, J.~Bialek, and J.~Bumby.
\newblock {\em Power system dynamics: stability and control}.
\newblock John Wiley \& Sons, 2008.

\bibitem{muller1972analysis}
P.C. M{\"u}ller and H.I. Weber.
\newblock Analysis and optimization of certain qualities of controllability and
  observability for linear dynamical systems.
\newblock {\em Automatica}, 8(3):237--246, 1972.

\bibitem{greedybound}
G.L. Nemhauser, L.A. Wolsey, and M.L. Fisher.
\newblock An analysis of approximations for maximizing submodular set
  functions.
\newblock {\em Mathematical Programming}, 14(1):265--294, 1978.

\bibitem{nipdirect}
M.~Nip, J.P. Hespanha, and M.~Khammash.
\newblock Direct numerical solution of algebraic lyapunov equations for
  large-scale systems using quantized tensor trains.

\bibitem{olshevsky2013minimal}
A.~Olshevsky.
\newblock The minimal controllability problem.
\newblock {\em arXiv preprint arXiv:1304.3071}, 2013.

\bibitem{tylerquote}
F.~Pasqualetti, S.~Zampieri, and F.~Bullo.
\newblock Controllability metrics and algorithms for complex networks.
\newblock {\em arXiv preprint arXiv:1308.1201}, 2013.

\bibitem{schrijver2003combinatorial}
A.~Schrijver.
\newblock {\em Combinatorial optimization: polyhedra and efficiency},
  volume~24.
\newblock Springer, 2003.

\bibitem{summers2013optimal}
T.H. Summers and J.~Lygeros.
\newblock Optimal sensor and actuator placement in complex dynamical networks.
\newblock {\em To appear, IFAC World Congress, Cape Town, South Africa}, 2014.

\bibitem{van2001review}
M.~Van De~Wal and B.~De~Jager.
\newblock A review of methods for input/output selection.
\newblock {\em Automatica}, 37(4):487--510, 2001.

\bibitem{balancedgramian}
K.~Zhou, G.~Salomon, and E.~Wu.
\newblock Balanced realization and model reduction for unstable systems.
\newblock {\em International Journal of Robust and Nonlinear Control},
  9(3):183--198, 1999.

\end{thebibliography}

\end{document}